\documentclass[12pt]{article}
\usepackage{amsmath,amssymb,amsthm,graphicx}
\usepackage{fullpage}
\usepackage{enumitem, verbatim}
\usepackage{xcolor}

 \newtheorem{thm}{Theorem}[section]
 \newtheorem{cor}[thm]{Corollary}
 \newtheorem{lem}[thm]{Lemma}
 \newtheorem{prop}[thm]{Proposition}
 \newtheorem{defn}[thm]{Definition}

 \newtheorem{ex}{Example}

\begin{document}

\title{Non-disjoint strong external difference families can have any number of sets}

\author{Sophie Huczynska* and Siaw-Lynn Ng**}
\date{* School of Mathematics and Statistics, University of St Andrews,\\
St Andrews, Fife, KY16 9SS,UK, sh70@st-andrews.ac.uk \\  ** Information Security Group, Royal Holloway, University of London, Egham, Surrey, TW20 0EX, UK, S.Ng@rhul.ac.uk}
\maketitle

\begin{abstract}
Strong external difference families (SEDFs) are much-studied combinatorial objects  motivated by an information security application. A well-known conjecture states that only one abelian SEDF with more than 2 sets exists.  We show that if the disjointness condition is replaced by non-disjointness, then abelian SEDFs can be constructed with more than 2 sets (indeed any number of sets). We demonstrate that the non-disjoint analogue has striking differences to, and connections with, the classical SEDF and arises naturally via another coding application.
\end{abstract}

\maketitle

\section{Introduction}

Difference families are long-studied combinatorial objects, with many applications.  A family of (not necessarily disjoint) $k$-sets in a group $G$ is a difference family if the multiset of pairwise internal differences (between distinct elements of each set) comprises each non-identity element of $G$ $\lambda$ times.  External difference families (EDFs) were introduced in \cite{OgaKurStiSai}, motivated by the information security problem of constructing optimal secret sharing schemes; the mathematical link between EDFs and algebraic manipulation detection (AMD) codes was formalized in \cite{PatSti}.  EDFs are a generalisation of difference families, in which the pairwise external differences (between distinct sets) are considered; the inter-set differences correspond to possible manipulations of an encoded message.  The sets are disjoint (to ensure unique decoding) and the multiset of external differences comprises each non-identity element of $G$ $\lambda$ times (the identity is ignored as it would correspond to ``no manipulation").  

An EDF is an optimal example of a weak AMD code (\cite{PatSti}).  There is a stronger security model (strong AMD code) which motivates the definition of a strong external difference family (SEDF) (\cite{PatSti}).  An SEDF is an EDF such that, for any set in the family, the pairwise differences between elements of this set and those of all other sets in the family comprise each non-identity element $\lambda$ times.  An SEDF is an EDF, but not every EDF is an SEDF.

Many constructions for SEDFs exist in the combinatorial literature (see \cite{BaJiWeZh, HucJefNep, JedLi, Wen}).  Surprisingly, all known infinite families have two sets, and just one example is known with more than two sets (having 11 sets in a group of size 243, obtained independently in \cite{JedLi} and \cite{Wen}).  It is conjectured in \cite{Leung1} that no other abelian SEDF exists with more than two sets; many theoretical results and computational searches support this (eg \cite{JedLi, Leung1, Leung2, MarSti}).  It is also notable that the only known infinite families with fixed $\lambda$ have $\lambda=1$.

In this paper we introduce the non-disjoint analogue of the SEDF (the sets are no longer disjoint and the multiset condition requires $\lambda$ occurences of every group element), and demonstrate that an abelian infinite family exists whose members can have any number of sets.  We also obtain infinite families of two-set non-disjoint SEDFs with any fixed frequency value $\lambda \in \mathbb{N}$.  For $\lambda=1$, we show that our two-set construction corresponds to a known two-set construction for classical SEDFs, and indicate why our non-disjoint constructions do not yield classical SEDFs with $\lambda>1$ or more than two sets.

The type of non-disjoint SEDFs which we construct satisfies a stronger condition on their external difference properties, namely that the multiset of external differences between any pair of distinct sets in the family comprises each element of $G$ precisely $\lambda$ times.  We call these \emph{pairwise strong external difference families} (PSEDFs).  Each PSEDF is a non-disjoint SEDF but not vice versa.   We demonstrate that PSEDFs are useful and indeed optimal for a different  application in communications theory, that of optical orthogonal codes and conflict-avoiding codes (\cite{ChuSalWei}, \cite{ShuWonChe}).  Here, the external differences between sets correspond to cross-correlation of binary sequences; there is no requirement for disjointness and the identity is treated in the same way as any other group element (since collision with the zero-shift of another sequence is just as significant as collision with any other shift).

\section{Background}
\label{sec:background}

All groups are written additively.  For two sets $A,B$ in a group $G$, define the multisets $\Delta(A,B)=\{a-b: a \in A, b \in B\}$ and $\Delta(A)=\{a_1-a_2: a_1,a_2 \in A, a_1 \neq a_2\}$. The notation $\lambda A$ denotes the multiset consisting of $\lambda$ copies of a set $A$. A translate of a set $A$ is denoted by $t+A=\{t+a:a \in A\}$.

The existing definition of an SEDF is as follows. 
\begin{defn}\label{def:SEDF}
Let $G$ be a group of order $v$ and let $m>1$.  A family of disjoint $k$-sets $\{A_1, \ldots, A_m\}$ in $G$ is a $(v,m,k,\lambda)$-SEDF if, for each $1 \leq i \leq m$, the multiset equation
$ \bigcup_{j \neq i} \Delta(A_i,A_j)= \lambda(G \setminus \{0\})$ holds.
\end{defn}

\begin{ex}\label{ex:PatSti}
$(\{0,1,\ldots,k-1\},\{k,2k,\ldots, k^2\})$ is a $(k^2+1,2,k,1)$-SEDF in $\mathbb{Z}_{k^2+1}$ (\cite{PatSti}).
\end{ex}

The following is a summary of known existence results for abelian SEDFs.
\begin{prop} \label{SEDFexistence}
A $(v,m,k, \lambda)$-SEDF exists in $G$ in the following cases:
\begin{itemize}
\item[(i)] $(v,m,k,\lambda)=(k^2+1, 2, k, 1)$, $G=\mathbb{Z}_{k^2+1}$ (\cite{PatSti}); 
\item[(ii)] $(v,m,k,\lambda)=(v, 2, \frac{v-1}{2}, \frac{v-1}{4})$, $v \, \equiv \, 1 \pmod 4$, $G$ contains a $(v,\frac{v-1}{2},\frac{v-5}{4},\frac{v-1}{4})$ partial difference set (\cite{HucPat}); 
\item[(iii)] $(v,m,k,\lambda)=(q,2,\frac{q-1}{4},\frac{q-1}{16})$, $q=16 t^2+1$ a prime power, $t \in \mathbb{Z}$, $G=(GF(q),+)$ (\cite{BaJiWeZh}); 
\item[(iv)] $(v,m,k,\lambda)=(p,2,\frac{p-1}{6},\frac{p-1}{36})$, $p=108 t^2+1$ is prime, $t \in \mathbb{Z}$, $G=(GF(p),+)$  (\cite{BaJiWeZh}).
\end{itemize}
\end{prop}

There are very many non-existence results for SEDFs in abelian groups with more than two sets (\cite{Leung1}); we summarize some key ones:
\begin{prop} Let $G$ be abelian.
No $(v,m,k,\lambda)$-SEDF with $m>2$ exists if:
\begin{itemize}
\item[(i)] $m \in \{3,4\}$ (\cite{MarSti}); or
\item[(ii)] $v$ is a product of distinct primes and $\gcd(mk,v)=1$ (\cite{BaJiWeZh}); or
\item[(iii)] $G$ is cyclic of prime power order \cite{Leung1}; or
\item[(iv)] $v$ is a product of at most three (not necessarily distinct) primes, except possibly when $G=C_p^3$ and $p$ is a prime greater than $3 \times 10^{12}$ \cite{Leung1}; or
\item[(v)] $G$ has order $p^2$ (\cite{Leung2}).
\end{itemize}
\end{prop}
Other non-existence conditions for $m>2$ include: $\lambda \geq k$  (\cite{HucPat}); $\lambda>1$ and $\frac{\lambda(k-1)(m-2)}{(\lambda-1)k(m-1)}>1$  (\cite{HucPat}); $k|v$ (\cite{MarSti}); $\gcd(k,v-1)=1$ (\cite{JedLi}) and $v-1$ is squarefree (\cite{HucJefNep}).  In \cite{Leung2}, a result is given for groups of order $pq$ for $p$ sufficiently large.

We adapt the classical definition of SEDF in Definition \ref{def:SEDF} by removing the disjointness condition, so the identity becomes a valid external difference.  For consistency we require that the identity occurs at the same frequency as the other group elements. This in fact means the sets \emph{must} be non-disjoint, so this structure is genuinely distinct from the classical SEDF, i.e. it is the non-disjoint analogue rather than a generalisation. 
\begin{defn}
Let $G$ be a group of order $v$ and let $m>1$.  We say that a family of $k$-sets $\{A_1, \ldots, A_m\}$ in $G$ is a \emph{non-disjoint $(v,m,k,\lambda)$-SEDF} if, for each $1 \leq i \leq m$, the multiset equation
$ \bigcup_{j \neq i} \Delta(A_i,A_j)= \lambda G$ holds.
\end{defn}
This may be viewed as mathematically more ``natural" than the existing SEDF definition, as it treats every element of the group equally.

A non-disjoint SEDF consisting of two sets $\{A,B\}$ satisfies the condition that $\Delta(A,B)=\Delta(B,A)=\lambda G$.  This motivates the following definition of a non-disjoint structure with a stronger condition on its external differences.

\begin{defn}
Let $G$ be a group of order $v$ and let $m>1$.  We say a family of $k$-sets $\{A_1, \ldots, A_m\}$ in $G$ is a $(v,m,k,\lambda)$-PSEDF if, for each $1 \leq i \neq j \leq m$, the multiset $\Delta(A_i,A_j)$
comprises $\lambda$ copies of each element of $G$.
\end{defn}

\begin{thm}
\begin{itemize}
\item[(i)] A $(v,m,k,\lambda)$-PSEDF is a non-disjoint $(v,m,k,(m-1)\lambda)$-SEDF.
\item[(ii)] A non-disjoint $(v,2,k,\lambda)$-SEDF is a $(v,2,k,\lambda)$-PSEDF.
\end{itemize}
\end{thm}

\begin{lem}\label{lem:relations}
\begin{itemize}
\item[(i)] For a non-disjoint $(v,m,k,\lambda)$-SEDF, $\lambda v=k^2(m-1)$.
\item[(ii)] For a $(v,m,k,\lambda)$-PSEDF, $\lambda v=k^2$.
\end{itemize}
\end{lem}

\begin{ex}
In $\mathbb{Z}_{18}$, the sets $\{0,1,2,3,4,5\}$ and $\{0,1,6,7,12,13\}$ form an $(18,2,6,2)$-PSEDF which is a non-disjoint $(18,2,6,2)$-SEDF.
\end{ex}

A classical $(v,m,k,1)$-SEDF exists in an abelian group if and only if $m=2$ and $v=k^2+1$ or $k=1$ and $m=v$ \cite{PatSti}.  An analogous result holds for non-disjoint SEDFs (for these, by Lemma \ref{lem:relations}(i), $k=1$ cannot occur).
\begin{thm}
In an abelian group $G$, a non-disjoint $(v,m,k,1)$-SEDF exists if and only if $m=2$ and $v=k^2$.
\end{thm}
\begin{proof}
Suppose $A_1,\ldots, A_m$ is a non-disjoint $(v,m,k,1)$-SEDF with $m \geq 3$. We have $k\geq 2$.  Then $\cup_{i \neq j} \Delta(A_i,A_j)=mG$; removing all differences to/from $A_1$, $\cup_{1 \neq i \neq j \neq 1} \Delta(A_i,A_j)=(m-2)G$.  Let $x,y \in A_1$, $x \neq y$.  Then $x-y$ must occur in $\Delta(A_i,A_j)$ for some $1 \neq i \neq j \neq 1$, i.e. $x-y=u-v (\neq 0)$ for some $u \in A_i$, $v \in A_j$, $i \neq j$.  But then $x-u=y-v$ occurs twice in $\cup_{k \neq 1} \Delta(A_1,A_k)$, a contradiction. Hence $m=2$. Theorem \ref{theorem:Ex4.8} establishes the reverse.
\end{proof}

When working in cyclic groups, we will use the following well-known correspondence with binary sequences (for background on sequences, see \cite[Section 5.4]{Men}). A \emph{binary sequence} of length $v$ is a sequence $X=x_0 \ldots x_{v-1}$ where each $x_i \in \{0,1\}$.  We also denote this by $X=(x_i)_{i=0}^{v-1}$.  We call a contiguous subsequence $x_{\delta} x_{\delta+1} \ldots x_{\delta+r-1}$ of $X$ a \emph{substring} of $X$ of length $r$. We take indices modulo $v$ unless otherwise stated.  The \emph{weight} of a binary sequence is the number of occurrences of the symbol $1$ in the sequence.  We call $X+s = (x_{i+s})_{i=0}^{v-1}$ a (cyclic) shift of $X$ by $s$ places.
\begin{defn}\label{def:set_sequence}
\begin{itemize}
\item[(i)] For a $k$-subset $A$ of $\mathbb{Z}_v=\{0\,\ldots,v-1\}$, we associate a binary sequence $X_A=(x_i)_{i=0}^{v-1}$ of weight $k$ whose $i$th  entry $x_i$ is $1$ if $i \in A$ and $0$ if $i \not\in A$.
\item[(ii)] For a binary sequence $X_A=(x_i)_{i=0}^{v-1}$ of weight $k$, we associate a $k$-subset $A$ of $\mathbb{Z}_v$, comprising all elements $i \in \{0,\ldots,v-1\}$ such that $x_i=1$.
\end{itemize}
\end{defn}
In $\mathbb{Z}_7$, the set $A=\{1,2,4\}$ corresponds to the sequence $0110100$.

Using this correspondence, we have the following useful relationship.

\begin{prop}\label{prop:basic}
Let $X_A=(x_i)_{i=0}^{v-1}$, $X_B=
(y_i)_{i=0}^{v-1}$ ($x_i, y_i \in \{0, 1\}$), with indices taken modulo $v$, be the sequences corresponding to $k$-subsets $A$ and $B$ in $\mathbb{Z}_v$.  Then:
\begin{itemize}
\item[(i)] For $\delta \in \{0,\ldots, v-1\}$, $\sum_{t=0}^{v-1} x_t y_{t+\delta}$ equals the number of occurrences of $\delta$ in $\Delta(B,A)$.
\item[(ii)] 
$\displaystyle{\sum_{t=0}^{v-1}} x_t y_{t+\delta} = \lambda$ for all $0 \le \delta \le v-1$ if and only if $\Delta(B,A)= \lambda \mathbb{Z}_n$
\end{itemize}
\end{prop}
\begin{proof}
For fixed $\delta \in \{0\,\ldots,v-1\}$ the sum $\sum_{t=0}^{v-1} x_t y_{t+\delta}$ counts the number of positions $t$ such that $x_t=1=y_{t+\delta}$.  This is the number of $t \in \mathbb{Z}_v$ such that $t \in A$ and $t+\delta \in B$, i.e. the number of times $\delta$ occurs in $\Delta(B,A)$.
\end{proof}

We end the section with some further sequence terminology. Let $X$ be a sequence. Following \cite{Men}, we define a \emph{run} of $X$ to be a substring of $X$ consisting of consecutive $0$s or consecutive $1$s which is neither preceded nor succeeded by the same symbol. We call a run of $0$s a \emph{gap} and a run of $1$s a \emph{block}.  For example in the length-$9$ sequence $111100010$, $1111$ is a substring which is a block of length $4$, and $000$ is a substring which is a gap of length $3$. 

\section{Constructions for PSEDFs}
\label{sec:PSEDFs}

In this section, we present results for PSEDFs and non-disjoint SEDFs which demonstrate their differences and similarities with classical SEDFs.  We use binary sequences.  For a sequence $X=(x_i)_{i=0}^{v-1}$, indices are taken modulo $v$.

We first construct infinite two-set families of non-disjoint SEDFs for any $\lambda$-value. For classical SEDFs, all known families with fixed $\lambda$ have $\lambda=1$.  

\begin{thm} \label{theorem:Ex4.8}
If $v|k^2$ and $k|v$ then
\begin{itemize}
\item[(i)] there exists an infinite family of $(v,2, k,\frac{k^2}{v})$-PSEDFs in $\mathbb{Z}_v$.
\item[(ii)] there exists an infinite family of non-disjoint $(v,2, k,\frac{k^2}{v})$-SEDFs in $\mathbb{Z}_v$.\\
Specifically, the sets of the PSEDF (and SEDF) are $A_X = \{0, 1, 2, \ldots, k-1\}$ and $A_Y= \left\{a k, ak+1, \ldots, ak+ \lambda-1 \; : \; a = 0, 1, \ldots, \left(\frac{v}{k}-1\right) \right\}$.

\end{itemize}
\end{thm}
\begin{proof}
Let $\lambda=\frac{k^2}{v}$. As $k \leq v$, we have $\lambda \leq k$.  The sequences corresponding to the sets $A_X$, $A_Y$ are:
$$
    X = \overbrace{11\ldots1}^k \overbrace{00\ldots0}^k \ldots \overbrace{00\cdots0}^k, \quad
    Y = \overbrace{\underbrace{1\ldots1}_{\lambda}0\ldots0}^k \overbrace{\underbrace{1\ldots1}_{\lambda}0\ldots0}^k \ldots \overbrace{\underbrace{1\ldots1}_{\lambda}0\ldots0}^k.$$
Here $X$ is a block of length $k$ then a gap of length $v-k$, while $Y$ comprises a block of length $\lambda$ then a gap of length $k-\lambda$, repeated $\frac{v}{k}$ times.  

Write $X=(x_t)$, $Y=(y_t)$.  Let $\delta \in \{0, \ldots, v-1\}$.  Consider $\Sigma_{t=0}^{v-1} x_t y_{t+\delta}$.  We see that $\Sigma_{t=0}^{v-1} x_t y_{t+\delta} = \Sigma_{t=0}^{k-1} x_t y_{t+\delta}$, since $x_t=0$ for $t=k, \ldots, v-1$, so we need only consider the length-$k$ substring $Y_{\delta}=y_{\delta} y_{1+\delta} \ldots y_{k-1+\delta}$ of $Y$.  The value of $\Sigma_{t=0}^{k-1} x_t y_{t+\delta}$ is exactly the number of 1s in $Y_{\delta}$.  By construction, if any length-$k$ substring $W$ of $Y$ starts with some $s \le \lambda$ 1s, it is followed by a gap of length $k-\lambda$, which is then followed by a block of length $\lambda-s$. 
If it starts with some $s \le k-\lambda$ 0s, it is followed by a block of length $\lambda$, which is then followed by a gap of length $k-\lambda-s$.   In either case there are always $\lambda$ 1s in $W$.  Hence $\Sigma_{t=0}^{v-1} x_t y_{t+\delta} = \Sigma_{t=0}^{k-1} x_t y_{t+\delta} = \lambda$.  This applies to any $\delta$, and hence by Proposition \ref{prop:basic}, $\{A_X,A_Y\}$ is a non-disjoint $(v,2,k, \frac{k^2}{v})$-PSEDF.
\end{proof}

\begin{cor}\label{cor:lambda}
\begin{itemize}
\item[(i)]  For any $a,r \in \mathbb{N}$, there exists a $(r a^2,2,ra,r)$-PSEDF in $\mathbb{Z}_{ra^2}$.
\item[(i)] Let $\lambda \in \mathbb{N}$.  Then in $\mathbb{Z}_{\lambda a^2}$, there exists a non-disjoint  $(\lambda a^2,2,\lambda a,\lambda)$-SEDF for all $a \in \mathbb{N}$. 
\item[(iii)] When $\lambda=1$, the sets $\{0,1,\ldots,k-1\}, \{0, k,2k,\ldots,(k-1)k\}$ form a non-disjoint $(k^2,2,k,1)$-SEDF in $\mathbb{Z}_{k^2}$.
\end{itemize}
\end{cor}
Note the similarity between the non-disjoint SEDFs of Corollary \ref{cor:lambda}(iii) and the SEDFs of Example \ref{ex:PatSti}; this will be explored subsequently.

\begin{ex}
\begin{itemize}
\item[(i)] In $\mathbb{Z}_9$, the sets $\{0,1,2\}$ and $\{0,3,6\}$ form a $(9,2,3,1)$-PSEDF corresponding to sequences $\{111000000, 100100100\}$.
\item[(ii)] In $\mathbb{Z}_8$, the sets  $\{0,1,2,3\}$ and $\{0,1,4,5\}$ form an $(8,2,4,2)$-PSEDF corresponding to sequences $\{11110000, 11001100\}$.
\end{itemize}
\end{ex}


We have the following generalisation of Theorem \ref{theorem:Ex4.8}.
\begin{thm}\label{thm:generalized}
Suppose $v|k^2$ and $k|v$.  Let $X=(x_t)_{i=0}^{v-1}$ be defined by $x_i=1$ for $0 \leq i \leq k-1$ and $x_i=0$ for $k \leq i \leq v-1$. If $Y=(y_t)_{t=0}^{v-1}$ is any sequence such that $(y_{t+k})=(y_t)$ and $y_0\ldots y_{k-1}$ has weight $\lambda=\frac{k^2}{v}$, then $\{X,Y\}$ corresponds to a $(v,2, k,\frac{k^2}{v})$-PSEDF in $\mathbb{Z}_v$ and a non-disjoint $(v,2, k,\frac{k^2}{v})$-SEDF in $\mathbb{Z}_v$.
\end{thm}

We next show non-disjoint SEDFs exist with any number of sets.
\begin{thm}\label{thm:manysets}
Let $N>1$.  There exists a $(2^N, N, 2^{N-1}, 2^{N-2})$-PSEDF in $\mathbb{Z}_{2^N}$.
\end{thm}
\begin{proof} 
For $1 \leq i \leq N$, define the binary sequence $X_i=(x_t)_{t=0}^{v-1}$ as follows:
\begin{align*}
    & x_0 = \cdots =x_{2^{i-1}-1} = 1, \\
    & x_{2^{i-1}} = \cdots = x_{2^{i}-1} = 0, \\
    & x_t = x_{t+2^i}, t \ge 2^i.
\end{align*}
So for each $X_i$ we have:
$$X_i = \underbrace{1\ldots1}_{2^{i-1}}\underbrace{0\ldots0}_{2^{i-1}} \underbrace{11\ldots100\ldots0}_{2^i} \ldots \underbrace{11\ldots100\ldots0}_{2^i}. $$

$X_i$ consists of a block of length $2^{i-1}$, followed by a gap of length $2^{i-1}$, and this length-$2^i$ substring is repeated $2^{N-i}$ times.  By construction, since $x_t=x_{t+2^i}$ for $t \geq 2^i$, every substring of length $2^i$ has an equal number of 1s and 0s, i.e. has weight $2^{i-1}$.  Therefore any substring of length $r$ where $2^i|r$ has weight $\frac{r}{2}$.  In particular, a substring of length $2^j$, $j \geq i$, has weight $2^{j-1}$.

We claim that these sequences $X_i$ ($1 \leq i \leq N$) correspond to sets $A_i$ ($1 \leq i \leq N$) which form a PSEDF in $\mathbb{Z}_{2^N}$ with the given parameters.  We determine $\Delta(A_i,A_j)$ for $1 \leq i \neq j \leq N$; by symmetry we may assume $i<j$.

Let $X_i =(z_t)$, $X_j=(y_t)$, with $i < j$.  Let $\delta \in \{0,1,\ldots, v-1\}$. We determine $S=\Sigma_{t=0}^{v-1} y_t z_{t+\delta}$.  Observe that $S = 2^{N-j} \times \Sigma_{t=0}^{2^{j}-1} y_t z_{t+\delta}$ since $z_t=z_{t+2^j}$ and $y_t=y_{t+2^j}$ ($t \geq 2^j$).  Moreover, since $y_{2^{j-1}}= \cdots =y_{2^j-1}=0$, $S=2^{N-j} \times \Sigma_{t=0}^{2^{j-1}-1} y_t z_{t+\delta}$. Since, from above, any substring of $X_i$ of length $2^{j-1}$ has weight equal to half its length, we have
$S= 2^{N-j} \times \frac{2^{j-1}}{2}= 2^{N-2}$.  Since $X_i$ and $X_j$ ($i<j$) were arbitrary (and using symmetry), we have that $\Delta(A_j,A_i)=2^{N-2}$ for all $i \neq j$.  Hence by Proposition \ref{prop:basic}, the sequences correspond to a $(2^N, N, 2^{N-1}, 2^{N-2})$-PSEDF in $\mathbb{Z}_{2^N}$ as required.
\end{proof}

The above theorem demonstrates a significant difference between the classical SEDF and its non-disjoint analogue:

\begin{cor}
For any $N>1$, there exists a non-disjoint $(2^N,N,2^{N-1},(N-1)2^{N-2})$-SEDF, i.e. a non-disjoint SEDF with $N$ sets.
\end{cor}

\begin{ex}
\begin{itemize}
\item[(i)] In $\mathbb{Z}_8$, the sets 
$\{0,2,4,6\}$, $\{0,1,4,5\}$ and $\{0,1,2,3\}$
form a $(8,3,4,2)$-PSEDF and disjoint $(8,3,4,4)$-SEDF, corresponding to sequences $10101010$, $11001100$ and $11110000$.
\item[(ii)]  In $\mathbb{Z}_{16}$, the sets 
$\{0,2,4,6,8,10,12,14\}$, $\{0,1,4,5,8,9,12,13\}$, \\ $\{0,1,2,3,8,9,10,11\}$ and $ \{0,1,2,3,4,5,6,7\}$
form a $(16,4,8,4)$-PSEDF and disjoint $(16,4,8,12)$-SEDF.
\end{itemize}
\end{ex}

\section{Relationship to classical SEDFs}
\label{sec:classical}
We next explain the similarity between the family of non-disjoint SEDFs in Corollary \ref{cor:lambda}(iii) and the family of SEDFs in Example \ref{ex:PatSti}. 

\begin{prop}\label{prop:extend}
Let $v|k^2$ and $k|v$. As subsets of $\mathbb{Z}_{v+1}$, the sets of the non-disjoint $(v,2, k,\frac{k^2}{v})$-SEDF in $\mathbb{Z}_v$ of Theorem \ref{theorem:Ex4.8} given by
\begin{itemize}
\item[] $A_{X'} = \{0, 1, 2, \ldots, k-1\}$; and 
\item[] $A_{Y'} = \left\{ak, ak+1, \ldots, ak+ \lambda-1 \; : \; a = 0, 1, \ldots, \left(\frac{v}{k}-1\right) \right\}$
\end{itemize}
satisfy $$\Delta(A_{X'}, A_{Y'})= \lambda(G \setminus \{v-k+1, \ldots, v-k+\lambda\})+ (\lambda-1) \{v-k+1,\ldots, v-k+\lambda\}.$$
\end{prop}
\begin{proof}
Take the two length-$v$ sequences $X,Y$ which correspond to the sets $A_X$ and $A_Y$ in Theorem \ref{theorem:Ex4.8}. By appending an additional $0$ at the end of each sequence, the new length-$(v+1)$ sequences $X',Y'$ correspond to the sets $A_{X'},A_{Y'}$ in $\mathbb{Z}_{k^2+1}$.   $X'$ is a block of length $k$ followed by a gap of length $v-k+1$, while $Y$ is a block of length $\lambda=k^2/v$ followed by a gap of length $k-\lambda$, which is repeated $v/k$ times, except that the final gap now has length $k-\lambda+1$.  Write $X'=(x'_t)$, $Y'=(y'_t)$.

Let $\delta \in \{0, \ldots, v\}$.  Consider $S=\Sigma_{t=0}^{v} x'_t y'_{t+\delta}$. As before, $S= \Sigma_{t=0}^{k-1} x'_t y'_{t+\delta}$, since $x'_t=0$ for $t=k, \ldots, v$, so we need only consider the length-$k$ substring $Y'_{\delta}=y'_{\delta} y'_{1+\delta} \ldots y'_{k-1+\delta}$ of $Y'$.  The value of $\Sigma_{t=0}^{k-1} x'_t y'_{t+\delta}$ is exactly the number of 1s in $Y'_{\delta}$, i.e. its weight.

We determine $Y'_0, \ldots, Y'_v$.  For $0 \leq \delta \leq v-1$, let $Y_{\delta}=y_{\delta} y_{1+\delta} \ldots y_{k-1+\delta}$, the substring of the original sequence $Y$.  We have $Y'_v = 0 y_0 \ldots y_{k-2}= Y_{v-1}$, and for $0 \leq \delta \leq v-k$, $Y'_{\delta}=Y_{\delta}$.

For the remaining $v-k+1 \leq \delta \leq v-1$, write $\delta=v-k+i$, $1 \leq i \leq k-1$.  In $Y$, the substring $Y_{v-k+i}=y_{v-k+i} y_{v-k+i+1} \ldots y_{i-2} y_{i-1}$.  The substring $Y'_{v-k+i}$ is obtained from $Y_{v-k+i}$ by inserting a $0$ between its $(k-1-i)$th and $(k-i)$th entries ($y_{v-1}$ and $y_0$), then deleting its final entry $y_{i-1}$.  Overall 
$$Y'_{v-k+i}=y_{v-k+i} y_{v-k+i+1} \ldots y_{v-1}0 y_0 \ldots y_{i-2}$$
with the entries after the zero being present only for $2 \leq i \leq k-1$.  Hence the overall change in symbols, going from $Y_{v-k+i}$ to $Y'_{v-k+i}$ ($1 \leq i \leq k-1$), is to replace $y_{i-1}$ by $0$.  Now, by construction $y_0= \cdots=y_{\lambda-1}=1$ and $y_{\lambda}= \cdots y_{k-1}=0$.  So, from the substrings $Y'_{v-k+i}$ with $1 \leq i \leq k-1$, only those with $i-1 \in \{0 \ldots, \lambda-1\}$ undergo a change in weight (a reduction by $1$ to weight $\lambda-1$); the rest have weight $\lambda$.

Hence by Proposition \ref{prop:basic}, $\Delta(A_{Y'},A_{X'})$ (and by symmetry $\Delta(A_{X'},A_{Y'})$) comprises $\lambda$ copies of $\mathbb{Z}_{v+1} \setminus \{v-k+1, \ldots, v-k+\lambda\}$ and $\lambda-1$ copies of $\{v-k+1, \ldots, v-k+\lambda\}$. 
\end{proof}

\begin{thm}
\begin{itemize}
\item[(i)] The non-disjoint $(k^2,2,k,1)$-SEDF in $\mathbb{Z}_{k^2}$ of Theorem \ref{theorem:Ex4.8} may be converted by set-translation to a classical $(k^2+1,2,k,1)$-SEDF in $\mathbb{Z}_{k^2+1}$.
\item[(ii)] Non-disjoint SEDFs of Theorem \ref{theorem:Ex4.8} with $\lambda>1$ cannot be so converted. 
\end{itemize}
\end{thm}
\begin{proof}
By Proposition \ref{prop:extend}, in $\mathbb{Z}_{k^2+1}$ the sets $A_{X'}=\{0,1,\ldots,k-1\}, A_{Y'}=\{0, k,2k,\ldots,(k-1)k\}$ satisfy $\Delta(A_{X'}, A_{Y'})= G \setminus \{v-k+1\}$.
Take the cyclic shift $Y''=Y'+(v-k+1)$ of $Y'$.  The set in $\mathbb{Z}_{k^2+1}$ corresponding to this new sequence is $A_{Y''}=\{k,2k,\ldots, k^2\}$, the translate $k+A_{Y'}$ of $A_{Y'}$ .  The pair $\{X,Y''\}$ correspond to sets $A_X, A_{Y''}$ which are disjoint and satisfy $\Delta(A_X, A_{Y''})=\lambda (G \setminus \{0\})$ with $\lambda=1$. This is the $(k^2,2,k,1)$-SEDF in $\mathbb{Z}_{k^2+1}$ of Example \ref{ex:PatSti}.
For (ii), observe the sets of Theorem \ref{theorem:Ex4.8} can be made disjoint by translation in $\mathbb{Z}_{v+1}$, only if the sequence $Y'$ has a gap of at least the size of the block in $X'$. This is possible only if $k-\lambda+1 \geq k$, i.e. $\lambda \leq 1$.
\end{proof}

Similarly, it is not possible to convert the non-disjoint SEDFs of Theorem \ref{thm:manysets} to classical SEDFs in $\mathbb{Z}_{2^N+1}$, except when $N=2$ (giving the $(5,2,2,1)$-SEDF $\{0,1\}, \{2,4\}$ in $\mathbb{Z}_5$).  For $N>2$, appending $0$ gives a structure with more than two frequencies and disjointness is impossible.
 
\section{Motivation from communications systems}
\label{sec:OOCs}

While classical SEDFs arise from AMD codes, non-disjoint SEDFs and PSEDFs have a different communications motivation.  \emph{Optical orthogonal codes} (OOCs) are sets of binary sequences with good auto- and cross-correlation properties for use in optical multi-access communication. The \emph{auto-correlation} of a sequence $X$ measures how much it collides with its shifts; its \emph{cross-correlation} with sequence $Y$ measures how much $X$ collides with the shifts of $Y$ (two sequences \emph{collide} in position $i$ if both have $1$'s in the $i$th position).

\begin{defn}
\label{defn:OOCsequences}
Let $v, w, \lambda_a, \lambda_c$ be non-negative integers with $v \ge 2$, $w \ge 1$.  Let ${\mathcal{C}}=\{X_0, \ldots, X_{N-1}\}$ be a family of $N$ binary sequences of length $v$ and weight $w$.  Then ${\mathcal{C}}$ is a $(v,w,\lambda_a, \lambda_c)$-OOC of size $N \ge 1$ if, writing $X=(x_i)_{i=0}^{v-1}$, $Y =(y_i)_{i=0}^{v-1}$ (indices modulo $v$): 
\begin{itemize}
\item[(i)] $\displaystyle{\sum_{t=0}^{v-1}} x_t x_{t+\delta} \le \lambda_a \mbox{ for any } X \in {\mathcal{C}}, 0 < \delta \le v-1$, and
\item[(ii)] $\displaystyle{\sum_{t=0}^{v-1}} x_t y_{t+\delta} \le \lambda_c \mbox{ for any } X, Y \in {\mathcal{C}}, 0 \le \delta \le v-1$,
\end{itemize}
i.e. if auto-correlation values are at most $\lambda_a$ and cross-correlation values are at most $\lambda_c$.
\end{defn}

Although called ``codes", OOCs are used as sets of periodic sequences, with $X_i$ being repeated.  A correlation value gets a contribution of $1$ precisely if both sequences have a 1 in the same position.  In using OOCs for communication, information can be sent only when there is a 1 in the sequence; if two sequences are used and there is a 1 in both sequences then interference occurs, which can result in errors in both received signals.  So a key design principle is to have low cross-correlation values. For more on OOCs see \cite{ChuSalWei}.  

By Definition \ref{def:set_sequence}, OOCs can be reformulated as subsets of $\mathbb{Z}_v$.  Let $\{ X_0,\ldots, X_{N-1}\}$ be a $(v,w,\lambda_a, \lambda_c)$-OOC.  For each sequence $X_i$, let $A_i$ be the set of integers modulo $v$ denoting the positions of the $1$s.  Then $A_i \subseteq \mathbb{Z}_v$, $|A_i| = w$ for all $0 \leq i \leq N-1$, and we have the conditions:
\begin{enumerate}
\item[(i)] $| A_i \cap (A_i + \delta) | \le \lambda_a$ for all $\delta \in
  \mathbb{Z}_v\setminus \{0\}$, i.e. any non-zero $\delta$ occurs in $\Delta(A_i)$ at most $\lambda_a$   times.
\item[(ii)] $| A_i \cap (A_j + \delta) | \le \lambda_c$ for all $\delta \in   \mathbb{Z}_v$, i.e. any  $\delta$ occurs in $\Delta(A_i, A_j)$ at most $\lambda_c$ times.
\end{enumerate}

An OOC with $\lambda_c=1$ and no auto-correlation requirement is a \emph{conflict-avoiding code} (CAC); see \cite{ShuWonChe}.  CACs are equivalently defined by the condition that $\{ \Delta(A_1), \ldots, \Delta(A_n)\}$ are pairwise disjoint (distinct $x_1,x_2 \in A_i$ and $y_1,y_2 \in A_j$ ($i \neq j$) with $x_1-x_2=y_1-y_2$ implies two distinct expressions for $x_1-y_1$ in $\Delta(A_i,A_j)$, and conversely).

\begin{prop} \label{prop:lcbound}
If $\mathcal{C}$ is a $(v,w, \lambda_a, \lambda_c)$-OOC with $|{\mathcal{C}}| \ge 2$, then
$\lambda_{c} \ge \frac{w^2}{v}$.
\end{prop}

\begin{proof}
Let $\mathcal{C}=\{A_0, \ldots, A_{N-1}\}$ as subsets of $\mathbb{Z}_v$. 
Let $F= \{((x, y), \delta) \; : \; x-y=\delta, x \in A_i, y \in A_j\}$ for some $A_i$, $A_j$, $A_i \neq A_j$.  There are $w$ values of $x$ and $w$ values of $y$, and for each pair of 
$(x,y)$ there is a unique $\delta=x-y$, so $|F| = w^2$.  On the other hand there
are $v$ possible values of $\delta$, and at most $\lambda_c$ pairs of $(x, y)$
such that $x-y = \delta$. So $|F| \le v \lambda_c$.  
\end{proof}

The lower bound is met when every $\delta$ occurs exactly ${w^2}/{v}$ times as an external difference.  If $\lambda_c = {w^2}/{v}$ for an OOC, then, since all cross-correlation values are at most $\lambda_c$ and at least ${w^2}/{v}$, we must have the cross-correlation equal to ${w^2}/{v}$ for all pairs of sequences, i.e. $\Delta(A_i, A_j) = \lambda_c \mathbb{Z}_v$ for all $A_i \neq A_j$.  Hence OOCs with cross-correlation values meeting the lower bound are in fact PSEDFs in $\mathbb{Z}_v$, and the PSEDFs in Section \ref{sec:PSEDFs} give examples of these OOCs.  In general, $(v, m, k, \lambda)$-PSEDFs in $\mathbb{Z}_v$ are $(v, k, \lambda_a, \lambda)$-OOCs for some $\lambda_a$, and $(v, k, \lambda_a, \lambda_c)$-OOCs are $(v,m, k, \lambda_c)$-PSEDFs if all cross-correlation values of the OOCs equal $\lambda_c$.  The extensions of PSEDFs from Proposition \ref{prop:extend} will give $(v+1, k, \lambda_a, \lambda_c)$-OOCs with $\lambda_c = \lceil \frac{k^2}{v+1} \rceil$, also best-possible. 

In OOC applications, auto-correlation aids synchronisation; minimising $\lambda_a$ is not a goal (\cite{ChuSalWei}). It is upper-bounded by the weight $w$ of the sequences; this bound is attained if a sequence $(x_i)_{i=0}^{v-1}$ satisfies $(x_{i+r})_{i=0}^{v-1}=(x_{i})_{i=0}^{v-1}$ for some $0<r<v$. In Theorem \ref{theorem:Ex4.8},  $\lambda_a=k$ since $Y$ satisfies $(y_{t+k})=(y_t)$ where $0<k<v$; in Proposition \ref{prop:extend},  $Y'$ has auto-correlation values strictly less than $k$, while $X'$ has auto-correlation exactly $k-1$, so $\lambda_a = k-1$.


\end{document}